\documentclass[11pt,reqno]{amsart}
\usepackage{enumerate}
\usepackage{amsrefs}
\usepackage{amsfonts, amsmath, amssymb, amscd, amsthm, bm, cancel,mathrsfs,simplewick}
\numberwithin{figure}{section}
\usepackage{url}
\usepackage{graphicx}
\usepackage[backref=false,linktocpage=true,colorlinks,citecolor=magenta,linkcolor=blue,urlcolor=magenta]{hyperref}
\usepackage{multicol}
\usepackage{comment}
\usepackage{tikz}
\usepackage{pgfplots}
\usepackage[margin=1.3in]{geometry}
\newtheorem{thm}{Theorem}[section]
\newtheorem*{thmA}{Theorem A}
\newtheorem*{thmB}{Theorem B}
\newtheorem*{thmC}{Theorem C}
\newtheorem*{thmD}{Theorem D}
\newtheorem*{thmE}{Theorem E}
\newtheorem{cor}[thm]{Corollary}

\newtheorem{lem}[thm]{Lemma}

\newtheorem{prop}[thm]{Proposition}
\theoremstyle{definition}

\newtheorem*{ack}{Acknowledgments}
\theoremstyle{remark}
\newtheorem{rem}[thm]{Remark}

\numberwithin{equation}{section}
\renewcommand{\(}{\left(}
\renewcommand{\)}{\right)}
\renewcommand{\~}{\tilde}

\renewcommand{\d}{\delta}

\renewcommand{\k}{\kappa}

\renewcommand{\t}{\theta}

\renewcommand{\L}{\Lambda}
\renewcommand{\hat}{\widehat}

\newcommand{\ra}{\rightarrow}

\newcommand{\vol}{\mathrm{vol}}
\newcommand{\Vol}{\mathrm{Vol}}

\begin{document}
\title[Affine isoperimetric type inequalities for static convex domains]{Affine isoperimetric type inequalities for static convex domains in hyperbolic space}

\author[Y. Hu]{Yingxiang Hu}
\address{School of Mathematical Sciences, Beihang University, Beijing 100191, P.R. China}
\email{\href{mailto:huyingxiang@buaa.edu.cn}{huyingxiang@buaa.edu.cn}}

\author[H. Li]{Haizhong Li}
\address{Department of Mathematical Sciences, Tsinghua University, Beijing 100084, P.R. China}
\email{\href{mailto:lihz@tsinghua.edu.cn}{lihz@tsinghua.edu.cn}}

\author[Y. Wan]{Yao Wan}
\address{Department of Mathematics, The Chinese University of Hong Kong, Shatin, Hong Kong 999077, P.R. China}
\email{\href{mailto:yaowan@cuhk.edu.hk}{yaowan@cuhk.edu.hk}}

\author[B. Xu]{Botong Xu}
\address{Department of Mathematics, Technion-Israel Institute of Technology, Haifa 32000, Israel}
\email{\href{mailto:botongxu@campus.technion.ac.il}{botongxu@campus.technion.ac.il}}

			
	\begin{abstract}
	In this paper, the notion of hyperbolic ellipsoids in hyperbolic space is introduced. Using a natural orthogonal projection from hyperbolic space to Euclidean space, we establish affine isoperimetric type inequalities for static convex domains in hyperbolic space. Moreover, equality of such inequalities is characterized by these hyperbolic ellipsoids.  
	\end{abstract}	
	\maketitle
	\section{Introduction}
	The classical Minkowski's second inequality \cite[Thm. 7.2.1]{Schneider} states that if $K\subset \mathbb R^n$ is a convex body (that is, a compact, convex set with non-empty interior) with smooth boundary $\partial K$ and $H$ is the mean curvature of $\partial K$, then 
	\begin{align}\label{s1:Minkowski-ineq}
	\operatorname{Area}(\partial K)^2 \geq ~\frac{n}{n-1}\Vol(K) \int_{\partial K} H dA,
	\end{align}
	with equality if and only if $K$ is a ball. This inequality was later generalized by Reilly \cite{Reilly1980} to compact Riemannian manifolds with nonnegative Ricci curvature and convex boundary. Using the generalized Reilly's formula \cite{Qiu-Xia2015}, Xia \cite{Xia2016} proved the following Minkowski type inequalities in hyperbolic space.	
\begin{thmA}[\cite{Xia2016}]
	Let $K$ be a smooth bounded domain in $\mathbb H^n$. Let $V(x)=\cosh r$, where $r(x)=d(x,p_0)$ is the geodesic distance to a fixed point $p_0\in \mathbb H^n$. Assume that the second fundamental form of $\partial K$ satisfies
	\begin{align}\label{s1:static-convex}
	h_{ij} \geq \frac{V_{,\nu}}{V} g_{ij}.
	\end{align}
	Then there holds
	\begin{align}\label{s1:Xia-Minkowski-ineq}
	\(\int_{\partial K} V dA\)^{2} \geq ~\frac{n}{n-1}\int_{K}V d\vol \cdot \int_{\partial K} V H dA,
	\end{align}
	where $d\vol$ is the volume element of $\mathbb H^n$. Equality holds if and only if $K$ is a geodesic ball.
\end{thmA}	
    The condition \eqref{s1:static-convex} is called {\em static convexity}, which was introduced by Brendle and Wang \cite{BW2014} for its correspondence in static space-time. In particular, when $K$ is a smooth bounded domain in $\mathbb R^n$, then the weight function $V\equiv 1$ and the static convexity \eqref{s1:static-convex} turns out to be the usual convexity (i.e. $h_{ij}\geq 0$), while the inequality \eqref{s1:Xia-Minkowski-ineq} reduces to the classical Minkowski's second inequality \eqref{s1:Minkowski-ineq}. Moreover, it was proved in \cite{Xia2016} that the equality in \eqref{s1:Xia-Minkowski-ineq} is attained for all geodesic balls, not necessarily the geodesic balls centered at $p_0$. 
    
%
Another family of sharp geometric inequalities for static convex domains in hyperbolic space was obtained by using locally constrained inverse curvature flows \cite{Hu-Li2022}.
\begin{thmB}[\cite{Hu-Li2022}]
	Let $K$ be a smooth bounded domain in $\mathbb H^{n}$. Assume that $\partial K$ is static convex and starshaped with respect to an interior point $p_0$ in $K$. For each $k=0,1,\ldots,n$, there holds
	\begin{align}\label{higher-order-ineq}
	\int_{\partial K} V H_k dA \geq &~ n\(\int_{K}V d\vol\)^\frac{n-1-k}{n}\(|\mathbb B^n|^\frac{2}{n}+\(\int_{K}V d\vol\)^\frac{2}{n}\)^\frac{k+1}{2},
	\end{align}
    where $H_k$ is the normalized $k$-th mean curvature of the hypersurface $\partial K$. 
	Equality holds in \eqref{higher-order-ineq} if and only if $K$ is a geodesic ball centered at $p_0$.
\end{thmB}
    For $k=0$, the inequality \eqref{higher-order-ineq} can be considered as a weighted isoperimetric inequality. It was first proved by Scheuer and Xia \cite{Scheuer-Xia2019}, and recently it was generalized to bounded domains with $C^1$ boundary by Li and Xu \cite{Li-Xu22}. For $k=1$, the inequality \eqref{higher-order-ineq} also holds for starshaped domains with mean convex boundary, see \cite{Scheuer-Xia2019}. Moreover, the inequality \eqref{higher-order-ineq} with $k=1$ also holds for merely static convex domains, i.e., $p_0$ is not necessarily an interior point of $K$, see \cite[Thm. 4, case 4]{BW2014}. Using the well-known Minkowski's identity
    \begin{align*}
    n\int_{K}V d\vol=\int_{\partial K} V_{,\nu} dA,
    \end{align*} 
    the inequality \eqref{higher-order-ineq} with $k=1$ can be rewritten as 
    \begin{align}\label{s1:Minkowski-ineq-equivalent-form}
    \int_{\partial K} H_1(\~\k) dA \geq &~n|\mathbb B^{n}|^\frac{2}{n}\(\int_{K}V d\vol\)^\frac{n-2}{n},
    \end{align}
    where $\~\k_i=V\k_i-V_{,\nu}$ and $H_1(\~\k)=VH_1-V_{,\nu}$.  This inequality \eqref{s1:Minkowski-ineq-equivalent-form} can be compared with the volumetric Minkowski inequality in Euclidean space, which holds for smooth bounded domain with mean convex boundary, see \cite[Thm. 1.5]{AFM2022}.
    
    As a natural analog of volume functional in Euclidean space, the weighted volume $\int_{K} V d\vol$ appears in the geometric inequalities in hyperbolic space such as \eqref{s1:Xia-Minkowski-ineq} and \eqref{higher-order-ineq}, as well as the Heintze-Karcher type inequalities in Riemannian manifolds \cite{Brendle2013} (see also \cites{Qiu-Xia2015,Li-Xia19}). On the other hand, different from the translation invariance of the Euclidean volume, the weighted volume depends on the choice of the origin $p_0$. Therefore, it is natural to classify the sharp geometric inequalities in hyperbolic space into the following two families. One family consists of the translation-invariant geometric inequalities, that is, the equality is attained by geodesic spheres, see e.g. \cites{ACW2021,AHL2020,Li-Wei-Xiong2014,Hu-Li-Wei2022,Wang-Xia2014,Ge-Wang-Wu2014,Hu-Li2019,Hu-Li2023}; The other one family of geometric inequalities depends on the choice of the origin $p_0$, and in this case the equality is attained by geodesic balls centred at $p_0$, see \cites{BHW2016,deLima-Girao2016,Ge-Wang-Wu2015,Hu-Li-Wei2022,Hu-Li2022,Li-Xu22-2,LW2023,Qiu-Xia2015,Xia2016,Hu-Wei-Zhou2023}. 

    On the other hand, affine isoperimetric inequalities relating two functionals associated with convex bodies such that the ratio of these functionals is invariant under the non-degenerate linear transformations. These inequalities are of great importance in both convex geometry and affine differential geometry, and they are more powerful than their Euclidean relatives. Moreover, due to the affine invariance, equality case of these inequalities is attained at ellipsoids.

The main purpose of this paper is to present affine isoperimetric type inequalities in hyperbolic space. First of all, we need to introduce the notion of hyperbolic ellipsoid as the natural counterpart of ellipsoid in Euclidean space. A smooth bounded domain $K$ in $\mathbb H^n$ is called a {\em hyperbolic ellipsoid} if there exists a point $p_0\in \mathbb H^n$ such that the image $\pi_{p_0}(K)$ is an ellipsoid in $\mathbb R^n$, where $\pi_{p_0}:\mathbb H^n \ra \mathbb R^n$ is an orthogonal projection with respect to $p_0$ given by \eqref{s2:defn-projection}. The {\em hyperbolic centroid} of a smooth bounded domain $K$ in $\mathbb H^n$ is the unique point defined by
	\begin{align}
	\operatorname{cen}(K)=\frac{\int_K X d\vol}{\(-\langle \int_K X d\vol,\int_K X d\vol \rangle\)^\frac{1}{2}},
	\end{align}
    see Proposition \ref{s2:thm-centroid-property}.	


The first result of this paper is the following geometric inequality in hyperbolic space, which can be considered as a natural counterpart of the classical affine isoperimetric inequality in Euclidean space.
\begin{thm}\label{s1:main-thm-1}
    Let $K$ be a smooth bounded domain in $\mathbb H^{n}$. Assume that the boundary $\partial K$ is static convex with respect to a point $p_0$ in $\mathbb H^n$. Then there holds
	\begin{align}\label{s1:hyperbolic-affine-isop-ineq}
	\int_{\partial K} H_{n-1}(\~\k)^\frac{1}{n+1} dA \leq n|\mathbb B^{n}|^\frac{2}{n+1} \(\int_{K}V d\vol\)^\frac{n-1}{n+1},
	\end{align}
	where $H_{n-1}(\~\k)=\prod_{i=1}^{n-1}(V\k_i-V_{,\nu})$. Equality holds if and only if $K$ is a hyperbolic ellipsoid. 
\end{thm}

Our second result is the following.
\begin{thm}\label{s1:main-thm-2}
	Let $K$ be a smooth bounded domain in $\mathbb H^{n}$ and $p_0$ is its hyperbolic centroid. Assume that the boundary $\partial K$ is static convex with respect to this point $p_0$. Then the following inequalities hold: 
	\begin{enumerate}[(i)]
		\item If $p>0$, then 
		\begin{align}\label{s1:Lp-affine-ineq-1}
		\int_{\partial K} (V_{,\nu})^\frac{(1-p)n}{n+p}H_{n-1}(\~\k)^{\frac{p}{n+p}} dA \leq n|\mathbb B^{n}|^{\frac{2p}{n+p}}\(\int_{K}V d\vol\)^{\frac{n-p}{n+p}}.
		\end{align}
		\item If $-n<p<0$, then 
		\begin{align}\label{s1:Lp-affine-ineq-2}
		\int_{\partial K} (V_{,\nu})^\frac{(1-p)n}{n+p}H_{n-1}(\~\k)^{\frac{p}{n+p}} dA \geq n|\mathbb B^{n}|^{\frac{2p}{n+p}}\(\int_{K}V d\vol\)^{\frac{n-p}{n+p}}.
		\end{align}
		\item If $p<-n$ and $\partial K$ is strictly static convex with respect to $p_0$, then 
		\begin{align}\label{s1:Lp-affine-ineq-3}
		\int_{\partial K} (V_{,\nu})^\frac{(1-p)n}{n+p}H_{n-1}(\~\k)^{\frac{p}{n+p}} dA \geq n c^\frac{np}{n+p}|\mathbb B^{n}|^{\frac{2p}{n+p}}\(\int_{K}V d\vol\)^{\frac{n-p}{n+p}},
		\end{align}
		where $c>0$ is a non-sharp constant.
	\end{enumerate}
   Equality holds in \eqref{s1:Lp-affine-ineq-1} or \eqref{s1:Lp-affine-ineq-2} if and only if $K$ is a hyperbolic ellipsoid.
\end{thm}

We should mention that the left-hand sides of these inequalities \eqref{s1:hyperbolic-affine-isop-ineq}--\eqref{s1:Lp-affine-ineq-3} are natural analogs of the affine surface area and $L_p$ affine surface area in Euclidean space, see \eqref{s3:hyperbolic-affine-support} and \eqref{s3:identity-affine-surface-area} in \autoref{sec:3}.

We also obtain the following Blaschke-Santal\'o type inequality in hyperbolic space, which can be considered as a natural analog of Blaschke-Santal\'o inequality in Euclidean space. To state our result, for any subset $K$ in $\mathbb H^n$, the {\em hyperbolic polar body} $K^{\circ}$ of $K$ with respect to $p_0$ is defined by
\begin{align}\label{s1:polar-hyperbolic}
K^{\circ}:=\bigcap_{Y\in K}\{X\in\mathbb H^{n}~|~\cosh d(X,p_0)\cosh d(Y,p_0)\leq \cosh d(X,Y) +1\}.
\end{align}
Different from the (usual) polar body of a strictly convex body in hyperbolic space which lies in the de Sitter space, if the smooth bounded domain $K$ is static convex with respect to $p_0$, then its hyperbolic polar body $K^{\circ}$ is also static convex with respect to $p_0$, see Corollary \ref{s2:cor-static-convex-equiv}. 

The following weighted volume product inequality in hyperbolic space can be compared with the volume product inequality in space forms \cite{Gao-Hug-Schneider03} (see also a different extension by Hu and Li \cite{Hu-Li2023}).
\begin{thm}\label{s1:main-thm-3}
	 Let $K$ be a bounded domain in $\mathbb H^n$ with smooth boundary $\partial K$. Assume that $\partial K$ is static convex with respect to its hyperbolic centroid $p_0$. Let $K^{\circ}$ be the hyperbolic polar body of $K$ with respect to $p_0$. Then there holds
	\begin{align}\label{s1:hyperbolic-BS-ineq}
	\int_{K}V d\vol\cdot \int_{K^{\circ}} V d\vol \leq |\mathbb B^{n}|^{2}.
	\end{align}
	Equality holds if and only if $K$ is a hyperbolic ellipsoid centred at $p_0$.
\end{thm}

\begin{rem}
Our results can be also established for static convex domains in the unit sphere, while in this case, $V(x)=\cos r(x)$ and $r(x)=d(x,p_0)$ is the geodesic distance from $x$ to $p_0$ in the sphere. 
\end{rem}

The paper is organized as follows. In \autoref{sec:2}, we use the hyperboloid model to define the orthogonal projection from hyperbolic space to Euclidean space. Using this orthogonal projection, we explore the relations between hypersurfaces in hyperbolic space and its orthogonal projection in Euclidean space. Moreover, the hyperbolic centroid of a smooth bounded domain in hyperbolic space are defined. In \autoref{sec:3}, we review the classical affine isoperimetric inequalities in Euclidean space. In order to establish their hyperbolic analogs, we introduce the hyperbolic affine transformations, hyperbolic $L_p$-affine surface area, hyperbolic affine support function and hyperbolic polar bodies. In  \autoref{sec:4}, we give the proofs of Theorems \ref{s1:main-thm-1}--\ref{s1:main-thm-3}. 

\begin{ack}
	This work was supported by National Key Research and Development Program of China 2021YFA1001800, NSFC Grant No.12101027, NSFC Grant No.12471047 and the Fundamental Research Funds for the Central Universities. The research leading to these results is part of a project that has received funding from the European Research Council (ERC) under the European Union's Horizon 2020 research and innovation programme (grant agreement No.101001677).
\end{ack}

\section{Preliminaries}\label{sec:2}
\subsection{Hyperboloid model of hyperbolic space}

The Minkowski space $\mathbb R^{n,1}$ is the $(n+1)$-dimensional vector space $\mathbb R^{n+1}$ equipped with the Minkowski inner product 
$$
\langle X,Y\rangle=\langle (x_0,x),(y_0,y)\rangle= -x_0y_0+\sum_{i=1}^n x_i y_i, 
$$ 
where $x=(x_1,\ldots,x_n)$, $y=(y_1,\ldots,y_n)\in \mathbb R^n$. The hyperboloid model of the hyperbolic space $\mathbb H^n$ is the upper sheet of the hyperboloid in the Minkowski space $\mathbb R^{n,1}$, that is,
\begin{align*}
\mathbb H^n=\{X=(x_0,x)\in \mathbb R^{n,1}~|~ \langle X,X\rangle =-1,~ x_0>0\}.
\end{align*}
For any point $X=(x_0,x)\in \mathbb H^n$, it follows from the relation $-1=\langle X,X\rangle =-x_0^2+|x|^2$ and $x_0>0$ that $X=(\sqrt{1+|x|^2},x)$. 

Let $N=(1,0)\in \mathbb H^{n}$. The {\em orthogonal projection} with respect to $N$ is defined by 
\begin{align}\label{s2:defn-projection-N}
\pi: \quad \quad \quad \mathbb H^n \quad \quad     &\ra \mathbb R^n, \nonumber\\
 (\sqrt{1+|x|^2},x) &\mapsto x,
\end{align}
see Figure \ref{s2:fig}.
\begin{figure}[htbp]\label{s2:fig}
	\begin{tikzpicture}[scale=1.1]
	\draw[red,thick,domain=-3:3] plot(\x,{sqrt{(\x*\x+1)}}) ;
	\draw (-3,0) -- (3,0);
	\draw[-stealth](-3,0)--(3,0)node[right]{$\mathbb R^n$};
	\draw[-stealth](0,0)--(0,3.5) node[left]{$x_0$};
	\draw[blue,thick,-stealth] (1,1.414) -- (1,0);
	\draw[dashed] (0,0) -- (3,3);
	\draw[dashed] (0,0) -- (-3,3);
	\fill (0,0) circle (1.5pt);
	\fill (0,1) circle (1.5pt);
	\fill (1,0) circle (1.5pt);
	\fill (2.5,2.5) circle (1.5pt);
	\fill (1,1.414) circle (1.5pt);
	\node at (0,1.2) {$(1,0)$};
	\node at (0,-0.2){$0$};
	\node at (1,-0.2){$x$};
	\node at (3,2.5){\small $(|y|,y)$};
	\node at (2,1.414){\small $(\sqrt{1+|x|^2},x)$};
	\node at (-1.9,3){$\mathbb H^n\subset \mathbb R^{n,1}$};
	\end{tikzpicture} 
	\caption{Orthogonal projection $\pi$}
\end{figure}
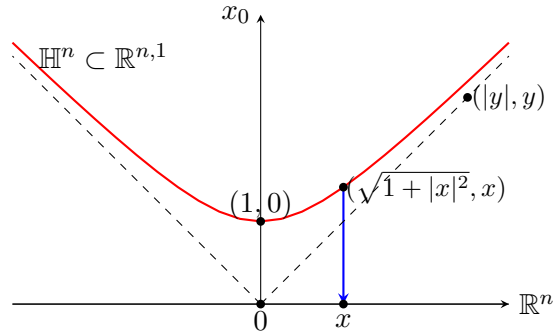
In general, for any point $p_0=(x_0,x)\in \mathbb H^n$, there exists a unique hyperbolic translation $T_{p_0}\in \operatorname{Isom}(\mathbb H^n)$ which sends $p_0$ to $N=(1,0)$. 
Then the orthogonal projection with respect to this point $p_0$ is given by 
\begin{align}\label{s2:defn-projection}
\pi_{p_0}=  \pi \circ T_{p_0}.
\end{align}

For any bounded domain $K$ with smooth boundary $\partial K$ in $\mathbb H^n$, the orthogonal projection $\hat{K}=\pi_{p_0}(K)$ with respect to a point $p_0\in \mathbb H^n$ is a smooth bounded domain in $\mathbb R^n$. This projection was first proposed by Gibbons \cite{Gibbons1997} to prove the Penrose inequality for general surfaces in the Minkowski spacetime. It is also called by {\em Gibbons' projection}, which can be used to define the quasi-local mass in general relativity \cites{Wang-Yau2009,Wang-Yau2009-2}. Brendle and Wang \cite{BW2014} also employed this projection to establish a Penrose type inequality for $2$-surfaces in a static convex timelike hypersurface in a static spacetime, see also \cite{Mars-Soria2014}. In particular, the static convexity of $K$ in this static spacetime implies the convexity of its projection $\hat{K}$ in $\mathbb R^{n}$, see \cite[P. 36]{BW2014}. In the case of hyperbolic space, we show that the static convexity of $\partial K$ with respect to $p_0$ is equivalent to the convexity of $\hat{\partial K}:=\pi_{p_0}(\partial K)$ in Euclidean space, see Corollary \ref{s2:cor-static-convex-equiv}. 

Using the geodesic polar coordinates with respect to $N=(1,0)$, the hyperbolic space $\mathbb H^n$ can be expressed as 
$$
X=(\sqrt{1+|x|^2},x)=(\cosh r, \sinh r\t), 
$$ 
where $\t\in \mathbb S^{n-1}$, $r$ is the geodesic distance to $N$ in $\mathbb H^n$. Denote by $\hat{r}=|x|$ the Euclidean distance of $x$ to the origin $o=(0,\ldots,0)$ in $\mathbb R^n$. It is clear that the hyperbolic radius $r$ and the Euclidean radius $\hat{r}$ are related by
\begin{align}\label{s2:radial-function-relation}
\hat{r}=\sinh r.
\end{align}
Let $\{\partial_1,\ldots,\partial_{n-1}\}$ be an orthonormal basis of the tangent space $T_x \hat{\partial K}$, and let $\hat{\nu}$ be the unit outward normal of $\hat{\partial K}$ at the point $x\in  \hat{\partial K}$, respectively. The induced metric and the support function of $\hat{\partial K}$ in $\mathbb R^n$ are given by
$$
\hat{g}_{ij}=\partial_i x \cdot\partial_j x=\d_{ij}, \quad  \hat{u}=x \cdot \hat{\nu},
$$
where $\cdot$ denotes the inner product of $\mathbb R^n$.
Then the second fundamental form $\hat{h}_{ij}$ and the Weingarten matrix $\hat{h}_i^j$ can be expressed by
$$
\hat{h}_{ij}=-\partial_i\partial_j x \cdot \hat{\nu}, \quad \hat{h}_i^j=\hat{g}^{jk}\hat{h}_{ki}.
$$

\begin{lem}\label{s2:key-lemma-1}
	Let $\partial K$ be a smooth hypersurface in $\mathbb H^n$. Let $\hat{\partial K}:=\pi(\partial K)$ be the orthogonal projection of $\partial K$ with respect to $N=(1,0)$ in $\mathbb R^n$. 
    \begin{enumerate}[(i)]
    	\item The induced metric $g_{ij}$ satisfies 
    	\begin{align}\label{s2:metric}
		g_{ij}=\d_{ij}-\frac{(x \cdot \partial_i x)(x \cdot\partial_j x)}{1+|x|^2}.
		\end{align}
		\item The unit outward normal $\nu$ satisfies 
		\begin{align}\label{s2:unit-normal}
		\nu=\frac{(\sqrt{1+|x|^2}\hat{u},\hat{u}x+\hat{\nu})}{(1+\hat{u}^2)^\frac{1}{2}}.
		\end{align}
		\item The support function $u$ satisfies 
		\begin{align} \label{s2:support-function}
		u=\(\frac{1+|x|^2}{1+\hat{u}^2}\)^\frac{1}{2}\hat{u}.
		\end{align}
		\item The second fundamental form $h=(h_{ij})$ satisfies 
		\begin{align}\label{s2:2nd-fundamental-form}	
		h_{ij}=\frac{1}{(1+\hat{u}^2)^\frac{1}{2}}\left[\hat{h}_{ij}+\hat{u}\d_{ij}-\frac{\hat{u}(x\cdot \partial_ix)(x\cdot \partial_j x)}{1+|x|^2}\right].
		\end{align}
		\item The Weingarten matrix $\mathcal{W}=(h_i^j)=(h_{ik}g^{kj})$ satisfies 
		\begin{align}\label{s2:Weingarten-matrix}
		h_i^j=\frac{\hat{h}_i^j+ \hat{u}\d_i^j}{(1+\hat{u}^2)^\frac{1}{2}}+\frac{\hat{h}_{i}^k(x\cdot\partial_k x)(x\cdot \partial_j x)}{(1+\hat{u}^2)^\frac{3}{2}}.
		\end{align}	
\end{enumerate}
\end{lem}
\begin{proof}
	(i) \eqref{s2:metric} follows from 
	\begin{align*}
	g_{ij}=&\langle \partial_i X,\partial_j X \rangle=\d_{ij}-\frac{(x \cdot \partial_i x)(x \cdot\partial_j x)}{1+|x|^2},
	\end{align*}
	where we used $\partial_i X=(\frac{x\cdot \partial_i x}{\sqrt{1+|x|^2}},\partial_i x)$.
	
	(ii) \eqref{s2:unit-normal} was proved in \cite[Lem. 2.2]{Li-Xu22}. 

    (iii) As $X=(\sqrt{1+|x|^2},x)=(\cosh r,\sinh r \t)$, we have
	\begin{align}\label{s2:radial-vector}
    \sinh r \partial_r X=\sinh r(\sinh r,\cosh r\t)=(|x|^2,\sqrt{1+|x|^2}x).
	\end{align}
    By \eqref{s2:radial-vector} and \eqref{s2:unit-normal}, we obtain 
	\begin{align*}
	u&=\langle \sinh r\partial_r X,\nu\rangle\\&=\left\langle(|x|^2,\sqrt{1+|x|^2} x),\frac{(\sqrt{1+|x|^2}\hat{u},\hat{\nu}+\hat{u}x)}{(1+\hat{u}^2)^\frac{1}{2}}\right\rangle\\
       &=\(\frac{1+|x|^2}{1+\hat{u}^2}\)^\frac{1}{2}\hat{u}.
	\end{align*}
	
    (iv) \eqref{s2:2nd-fundamental-form} follows from
	\begin{align*}
	h_{ij}&=-\langle \partial_i\partial_j X,\nu\rangle   \nonumber\\
	&=-\left\langle (\partial_i\partial_j \sqrt{1+|x|^2},\partial_i\partial_j x),\frac{(\sqrt{1+|x|^2}\hat{u},\hat{\nu}+\hat{u}x)}{(1+\hat{u}^2)^\frac{1}{2}}\right\rangle\nonumber\\
	&=\frac{\hat{h}_{ij}}{(1+\hat{u}^2)^\frac{1}{2}}-\frac{\hat{u}(\partial_i\partial_j x\cdot x)}{(1+\hat{u}^2)^\frac{1}{2}}+\frac{\hat{u}\sqrt{1+|x|^2}}{(1+\hat{u}^2)^\frac{1}{2}}\partial_i\partial_j \sqrt{1+|x|^2}\nonumber\\
	&=\frac{1}{(1+\hat{u}^2)^\frac{1}{2}}\left[\hat{h}_{ij}+\hat{u}\d_{ij}-\frac{\hat{u}(x\cdot \partial_ix)(x\cdot \partial_j x)}{1+|x|^2}\right].
	\end{align*}
	
	(v) Using $g^{ik}g_{kj}=\d_j^i$, it is easy to deduce that
	\begin{align*}
	g^{ij}=\d^{ij}+\frac{(x \cdot\partial_i x)(x\cdot\partial_j x)}{1+\hat{u}^2}.
	\end{align*}
	Then we have
	\begin{align*}
	h_i^j=g^{jk}h_{ki}=&\frac{\hat{h}_i^j+\hat{u}\d_i^j}{(1+\hat{u}^2)^\frac{1}{2}}+\frac{\hat{h}_{i}^k(x\cdot\partial_k x)(x\cdot \partial_j x)}{(1+\hat{u}^2)^\frac{3}{2}}.
	\end{align*}
\end{proof}

Using Lemma \ref{s2:key-lemma-1}, we have 
\begin{lem}\label{s2:key-lemma-2}
	Let $d\vol$ (resp. $d\hat{\vol}$) and $dA$ (resp. $d\hat{A}$) be the volume element and surface area element of $K\subset \mathbb H^n$ (resp. $\hat{K}=\pi(K)\subset \mathbb R^n$). 
	\begin{enumerate}[(i)]
	\item The volume elements of $K$ and $\hat{K}$ satisfy
	\begin{align}
	V d\vol=&~ d\hat{\vol}. \label{s2:weighted-volume}
	\end{align}
	\item The surface area elements of $K$ and $\hat{K}$ satisfy
	\begin{align}
	dA=& \(\frac{1+\hat{u}^2}{1+|x|^2}\)^\frac{1}{2}d\hat{A}. \label{s2:area}
	\end{align}
	\item The second fundamental forms of $\partial K$ and $\widehat{\partial K}$ satisfy
	\begin{align}\label{s2:shifted-second-fundamental-form}
	V h_{ij}-V_{,\nu} g_{ij}=&\(\frac{1+|x|^2}{1+\hat{u}^2}\)^\frac{1}{2}\hat{h}_{ij}.
	\end{align}
	\item The Gauss curvatures of $\partial K$ and $\widehat{\partial K}$ satisfy
	\begin{align}
	 H_{n-1}(\~\k)=&\(\frac{1+|x|^2}{1+\hat{u}^2}\)^\frac{n+1}{2}H_{n-1}(\hat{\k}),\label{s2:Gauss-curvature}
	\end{align}
	where $\~\k_i=V\k_i-V_{,\nu}$, and $\k_i$ (resp. $\hat{\k}_i$) are the principal curvatures of $\partial K\subset \mathbb H^n$ (resp. $\hat{\partial K}\subset \mathbb R^n$). 
\end{enumerate}
\end{lem} 
\begin{proof}	
	Formulas \eqref{s2:weighted-volume} and \eqref{s2:area} have been proved in \cite[Lem. 2.1 \& 2.3]{Li-Xu22}. 
    To deduce \eqref{s2:shifted-second-fundamental-form}, it follows from \eqref{s2:metric}, \eqref{s2:support-function} and \eqref{s2:2nd-fundamental-form} that
	\begin{align*}
	Vh_{ij}-V_{,\nu}g_{ij}&=\sqrt{1+|x|^2}h_{ij}-\frac{\sqrt{1+|x|^2}}{(1+\hat{u}^2)^\frac{1}{2}}\hat{u}g_{ij}\\
 &=\(\frac{1+|x|^2}{1+\hat{u}^2}\)^\frac{1}{2}\hat{h}_{ij}.
	\end{align*}
    Finally, by using \eqref{s2:metric}, we have 
    \begin{align*}
    \det(g_{ij})=1 -\frac{|x^\top|^2}{1+|x|^2}=\frac{1+|\hat{u}|^2}{1+|x|^2}.
    \end{align*}
    Then \eqref{s2:Gauss-curvature} follows from
	\begin{align*}
	H_{n-1}(\~\k)&=\frac{\det(Vh_{ij}-V_{,\nu}g_{ij})}{\det(g_{ij})}\\
      &=\(\frac{1+|x|^2}{1+\hat{u}^2}\)^\frac{n+1}{2}\frac{\det(\hat{h}_{ij})}{\det(\hat{g}_{ij})}\\
      &=\(\frac{1+|x|^2}{1+\hat{u}^2}\)^\frac{n+1}{2}H_{n-1}(\hat{\k}).
	\end{align*}
\end{proof}

\begin{cor}\label{s2:cor-static-convex-equiv}
	A smooth bounded domain $K$ in $\mathbb H^{n}$ is (resp. strictly) static convex with respect to $p_0$ if and only if its orthogonal projection $\hat{K}=\pi_{p_0}(K)$ in $\mathbb R^n$ is (resp. strictly) convex.
\end{cor}
\begin{proof}
For any $p_0 \neq N$, we use the hyperbolic translation $T_{p_0}$ such that $T_{p_0}(p_0)=N$. Then by $\pi_{p_0}=\pi\circ T_{p_0}$ due to \eqref{s2:defn-projection}, we have $T_{p_0}(K)$ is static convex with respect to $T_{p_0}(p_0)=N$. Moreover, it follows from \eqref{s2:shifted-second-fundamental-form} that $T_{p_0}(K)$ is static convex with respect to $T_{p_0}(p_0)=N$ if and only if $\hat{K}=\pi_{p_0}(K)=\pi \circ T_{p_0}(K)$ is convex in $\mathbb R^n$.
\end{proof}

\subsection{Hyperbolic centroid}\label{sec:2.3}
Recall that the {\em centroid} of a smooth bounded domain $\hat{K}$ in $\mathbb R^n$ is defined by
\begin{align}\label{s2:centroid-Euclidean-space}
\operatorname{cen}(\hat{K})=\frac{\int_{\hat K} x d\hat{\vol}}{\Vol(\hat{K})} \in \mathbb R^n,
\end{align}
where $x$ is the position vector of the points in $\hat{K}$. It is easy to see that
\begin{align*}
\int_{\hat{K}}\left|x-\operatorname{cen}(\hat{K})\right|^2 d\hat{\vol}(x)=\min_{y\in \mathbb R^n}\left\{\int_{\hat{K}}|x-y|^2 d\hat{\vol}(x)\right\}.
\end{align*}
Motivated by this, the {\em hyperbolic centroid} of a smooth bounded domain $K$ in $\mathbb H^n$ is defined by
\begin{align}\label{s2:centroid-hyperbolic-space}
\operatorname{cen}(K)=\frac{\int_K X d\vol}{\(-\langle \int_K X d\vol,\int_K X d\vol\rangle\)^\frac{1}{2}},
\end{align}
where $X=(\sqrt{1+|x|^2},x)$ is the position vector of the points in $K$. By this definition, we know that $\operatorname{cen}(K)\in \mathbb H^n$ since $\langle \operatorname{cen}(K),\operatorname{cen}(K)\rangle=-1$ and the $0$th component of $\operatorname{cen}(K)$ is positive. 
\begin{prop}\label{s2:thm-centroid-property}
	The hyperbolic centroid of a smooth bounded domain $K$ in $\mathbb H^n$ satisfies 
	\begin{align}\label{s3:centroid-equivalent-defn}
	\int_{K}\cosh d(\operatorname{cen}(K),X) d\vol(X)=\min_{Y\in \mathbb H^n}\left\{\int_{K} \cosh d(Y,X) d\vol(X)\right\}.
	\end{align}
	Moreover, for any $A\in \operatorname{Isom}(\mathbb H^n)$, we have
	\begin{align}\label{s3:centroid-isometry}
	\operatorname{cen}(A(K))=A(\operatorname{cen}(K)).
	\end{align}
\end{prop}
\begin{proof}
	(i) First, we verify the equivalent characterization \eqref{s3:centroid-equivalent-defn} of the hyperbolic centroid. For any $Y\in \mathbb H^n$, we have
	\begin{align}
	\int_K \cosh d(Y,X) d\vol(X)=&-\int_{K} \langle Y,X\rangle d\vol(X) \nonumber\\
	=&-\langle Y, \operatorname{cen}(K)\rangle \(-\langle \int_K X d\vol,\int_K X d\vol\rangle\)^\frac{1}{2} \nonumber\\
	\geq & \(-\langle \int_K X d\vol,\int_K X d\vol\rangle\)^\frac{1}{2}, \label{s3:centroid-ineq}
	\end{align}
	where we used $-\langle Y, \operatorname{cen}(K)\rangle=\cosh d(Y,\operatorname{cen}(K))\geq 1$. Equality holds in \eqref{s3:centroid-ineq} if and only if $Y=\operatorname{cen}(K)$. 
	
	(ii) For any $A\in \operatorname{Isom}(\mathbb H^n)$ and any $X,Y\in \mathbb H^n$, we have $-\langle AX,AY \rangle= \cosh d(AX,AY)=\cosh d(X,Y)=-\langle X,Y \rangle$ and hence
	\begin{align}\label{s2:invariance-of-int-X}
	&-\langle \int_K AX d\vol(X), \int_K AX d\vol(X) \rangle \nonumber\\
	=& -\int_K\int_K \langle AX,AY \rangle d\vol(X) d\vol(Y) \nonumber\\ 
	=&-\int_K\int_K \langle X,Y \rangle d\vol(X) d\vol(Y) \nonumber\\ 
	=&-\langle \int_K X d\vol(X), \int_K X d\vol(X) \rangle. 
	\end{align} 
	Therefore, $-\langle \int_K X d\vol(X), \int_K X d\vol(X) \rangle$ is invariant under the isometry of $\mathbb H^n$. By \eqref{s2:centroid-hyperbolic-space} and \eqref{s2:invariance-of-int-X}, we have
	\begin{align*}
	\operatorname{cen}(A(K))=&\frac{\int_K AX d\vol}{\(-\langle \int_K AX d\vol,\int_K AX d\vol\rangle\)^\frac{1}{2}}\\
	       =&A\(\frac{\int_K X d\vol}{\(-\langle \int_K X d\vol,\int_K X d\vol\rangle\)^\frac{1}{2}}\)=A(\operatorname{cen}(K)).
	\end{align*}
\end{proof}

The hyperbolic centroid of $K$ in $\mathbb H^n$ and the centroid of $\hat{K}=\pi(K)$ in $\mathbb R^n$ can be related as follows.
\begin{prop}\label{s2:thm-centroid-characterization}
	A smooth bounded domain $K$ in $\mathbb H^n$ has its hyperbolic centroid at $p_0$ if and only if $\hat{K}=\pi_{p_0}(K)\subset \mathbb R^{n}$ has its centroid at $o$. 
\end{prop}
\begin{proof}
	Up to a hyperbolic translation, one may assume that the hyperbolic centroid of $K$ is $N=(1,0)$. Then we have $X=(\sqrt{1+|x|^2},x)=(\cosh r,\sinh r \t)$ and $x=\hat{r}\t$, where $\sinh r=\hat{r}$ and $\theta \in \mathbb S^{n-1}$. A direct calculation yields 
	\begin{align*}
	 \left.\nabla^{\mathbb H^n}\right|_{(1,0)} \int_K \cosh r d\vol=&\int_K d\pi(\sinh r \partial_r) d\vol\\
	=&\int_{K} \sinh r \cosh r\t (\sinh r)^{n-1} dr d\t\\
	=&\int_{\hat{K}} \hat{r}^n \t d\hat{r} d\t\\
	=&\int_{\hat{K}} x d\hat{\vol}.
	\end{align*}
	By the equivalent characterization \eqref{s3:centroid-equivalent-defn} of the hyperbolic centroid $\operatorname{cen}(K)$, we conclude that $o$ is also the centroid of $\hat{K}=\pi(K)$ in $\mathbb R^n$.
	
\end{proof}

\section{Affine isoperimetric inequalities}\label{sec:3}
In this section, we first recall the classical affine isoperimetric inequalities in Euclidean space. A compact convex subset of $\mathbb R^n$ with nonempty interior is called a {\em convex body} in $\mathbb R^n$. Denote by $\mathcal{K}^n$ the set of all convex bodies in $\mathbb R^n$. We denote by $\mathcal{K}^n_o\subset\mathcal{K}^n$ be the compact convex bodies in $\mathbb R^n$ containing the origin $o$ as an interior point.

Let $\hat{K}\in \mathcal{K}^n$. The {\em affine surface area} $\operatorname{as}(\hat{K})$ is defined by
\begin{align}\label{s1:def-affine-surf-area}
\operatorname{as}(\hat{K}):=\int_{\hat{\partial K}} H_{n-1}(\hat{\k})^\frac{1}{n+1} d\hat{A},
\end{align}  
where $H_{n-1}(\hat{\k})$ is the generalized Gauss curvature of the hypersurface $\hat{\partial K}\subset \mathbb R^{n}$.
Denote by $\operatorname{SL}(n)=\{ L\in M_{n\times n}~|~\det(L)=1\}$ the special linear group on $\mathbb R^n$, which consists of the unimodular linear transformations of $\mathbb R^n$. For any ellipsoid $E$ in $\mathbb R^n$, there exists a centered ball $B_E$ with $\Vol(B_E)=\Vol(B)$, $L\in \operatorname{SL}(n)$ and $b\in \mathbb R^n$ such that $E=L B_E+b$. The affine surface area $\operatorname{as}(\hat{K})$ is $\operatorname{SL}(n)$-invariant and translation-invariant, i.e., 
$$
\operatorname{as}(L(\hat{K})+b)=\operatorname{as}(\hat{K}), \quad \text{for all $L\in \operatorname{SL}(n)$ and $b\in \mathbb R^n$.}
$$ 

\begin{thmC}
Let $\hat{K}\in \mathcal{K}^n$. The classical affine isoperimetric inequality is 
\begin{align}\label{s3:affine-iso-ineq}
\operatorname{as}(\hat{K}) \leq n|\mathbb B^{n}|^\frac{2}{n+1} \Vol(\hat{K})^\frac{n-1}{n+1},
\end{align}
which is $\operatorname{SL}(n)$-invariant and translation-invariant. Equality holds if and only if $\hat{K}$ is an ellipsoid. 
\end{thmC}
For a more restricted class of bodies, this inequality is due to Blaschke \cite{Blaschke1916} for $n\leq 3$, see also \cite{Blaschke1923}. Later, Santal\'o \cite{Santalo1949} and Deicke \cite{Deicke1953} extended this inequality for all $n\geq 2$. Theorem C is due to Petty \cite{Petty1985}. For $\hat{K}\in \mathcal{K}^n$ with its centroid at the origin $o$, the affine surface area can be generalized to the $L_p$-affine surface area, which was first introduced by Lutwak \cite{Lutwak1996} for $p>1$ in his groundbreaking paper, and later generalized by Sch\"utt and Werner \cite{Schutt-Werner2004} for all $p\neq -n$. Precisely, for $p\neq -n$, the $L_p$-affine surface area $\operatorname{as}_p(\hat{K})$ of $\hat{K}\subset \mathbb R^{n}$ is defined by 
\begin{align}\label{def-as-p}
\operatorname{as}_p(\hat{K}):=\int_{\hat{\partial K}} \hat{u}^{-\frac{n(p-1)}{n+p}} H_{n-1}(\hat{\k})^\frac{p}{n+p} d\hat{A},
\end{align}
while for $p=\pm \infty$, it is defined by 
\begin{align}
\operatorname{as}_{\pm\infty}(\hat{K}):=\int_{\hat{\partial K}} \hat{u}^{-n} H_{n-1}(\hat{\k}) d\hat{A},
\end{align}
provided that the above integrals exist.

The following $L_p$ affine isoperimetric inequalities were proved by Lutwak \cite{Lutwak1996} for $p>1$, and later extended by Werner and Ye \cite[Thm. 4.2]{Werner-Ye2008} for general $p\neq -n$.
\begin{thmD}
	Let $\hat{K}\in \mathcal{K}^n$ with its centroid at the origin $o$.
	\begin{enumerate}[(i)]
		\item If $p>0$, then
		\begin{align}\label{s2:affine-ineq-1}
		\operatorname{as}_p(\hat{K}) \leq n|\mathbb B^{n}|^\frac{2p}{n+p}\Vol(\hat{K})^\frac{n-p}{n+p}.
		\end{align}
		\item If $-n<p<0$, then 
		\begin{align}\label{s2:affine-ineq-2}
		\operatorname{as}_p(\hat{K}) \geq n|\mathbb B^{n}|^\frac{2p}{n+p}\Vol(\hat{K})^\frac{n-p}{n+p}.
		\end{align}
		\item If $p<-n$ and $\hat{K}$ is strictly convex with $C^2$ boundary, then
		\begin{align}\label{s2:affine-ineq-3}
		\operatorname{as}_p(\hat{K}) \geq  nc^{\frac{np}{n+p}}|\mathbb B^{n}|^\frac{2p}{n+p}\Vol(\hat{K})^\frac{n-p}{n+p}.
		\end{align}
		where $c$ is a non-sharp constant.
	\end{enumerate}
    Equality holds in \eqref{s2:affine-ineq-1} or \eqref{s2:affine-ineq-2} if and only if $\hat{K}$ is an ellipsoid centered at the origin.
\end{thmD}
\begin{rem}
	Inequalities \eqref{s2:affine-ineq-1} - \eqref{s2:affine-ineq-3} are $\operatorname{SL}(n)$-invariant. Moreover, the constant $c$ in \eqref{s2:affine-ineq-3} is from the inverse Santal\'o inequality (see \cite{Bourgain-Milman1987}).
\end{rem}

Let $\hat{K}$ be a smooth convex body in $\mathbb R^n$ with strictly convex boundary. The {\em affine support function} of $\hat{K}$ is defined by
$$
\L(\hat{K}):=\hat{u}H_{n-1}(\hat{\k})^{-\frac{1}{n+1}},
$$ 
where $H_{n-1}(\hat{\k})$ denotes the Gauss curvature of $\hat{\partial K}$. Then $\L(\hat{K})$ is $\operatorname{SL}(n)$-invariant.
The {\em affine surface area measure} is defined by
\begin{align*}
    d\~A=H_{n-1}(\hat{\k})^\frac{1}{n+1}d\hat{A},
\end{align*} 
which is the volume element of the positive definite tensor $\~g_{ij}=H_{n-1}(\hat{\k})^{-\frac{1}{n+1}}h_{ij}$. Then $d\~A$ is $\operatorname{SL}(n)$-invariant and translation-invariant. 

For $p\neq -n$, it follows from the definition \eqref{def-as-p} that 
\begin{align*}
\operatorname{as}_p(\hat{K})=\int_{\hat{\partial K}}\(\frac{\hat{u}}{H_{n-1}(\hat{\k})^\frac{1}{n+1}}\)^\frac{(1-p)n}{n+p} H_{n-1}(\hat{\k})^\frac{1}{n+1}d\hat{A}
=&\int_{\hat{\partial K}} \L^\frac{(1-p)n}{n+p} d\~A.
\end{align*}
In particular, $\operatorname{as}_0(\hat{K})=\frac{1}{n}\Vol(\hat{K})$ and $\operatorname{as}_1(\hat{K})=\operatorname{as}(\hat{K})$ are $\operatorname{SL}(n)$-invariant and translation-invariant; for other choice of $p\neq -n$, $\operatorname{as}_p(\hat{K})$ is only $\operatorname{SL}(n)$-invariant.

The {\em affine transformation} in $\mathbb R^n$ is given by
\begin{align*}
\hat{\phi}_{L,b}: \mathbb R^n \ra \mathbb R^n,~ x\mapsto Lx+b,
\end{align*} 
where $L\in \operatorname{SL}(n)$ and $b\in \mathbb R^n$. Motivated by this, we define the {\em hyperbolic affine transformation} on $\mathbb H^n$ by
\begin{align*}
\phi_{L,b}: \mathbb H^n \ra \mathbb H^n,~ (\sqrt{1+|x|^2},x) \mapsto (\sqrt{1+|Lx+b|^2},Lx+b).
\end{align*}
In particular, we have $\phi_{L,b}(K)=\pi_{p_0}^{-1}\left(L(\pi_{p_0}(K))+b\right)$, where $p_0\in \mathbb H^n$ is a fixed point. Then the {\em hyperbolic special affine group} is defined by 
$$
\operatorname{SA}(\mathbb H^n):=\{\phi_{L,b} ~|~L\in \operatorname{SL}(n),b\in\mathbb R^n\},
$$
while the {\em hyperbolic special linear group} is defined by
$$
\operatorname{SL}(\mathbb H^n):=\{\phi_{L,0} ~|~L\in \operatorname{SL}(n)\}.
$$

Let $K$ be a smooth bounded domain in $\mathbb H^n$, which is static convex with respect to its hyperbolic centroid $p_0$. Then $\hat{K}:=\pi_{p_0}(K)$ is a convex body with its centroid at the origin $o$. The {\em hyperbolic affine support function} $\L^{\mathbb H}$ of $K\subset \mathbb H^n$ is defined by
\begin{align}\label{s3:affine-support-function}
\L^{\mathbb H}(K):=u H_{n-1}(\~\k)^{-\frac{1}{n+1}}.
\end{align}
It follows from \eqref{s2:support-function} and \eqref{s2:Gauss-curvature} that 
\begin{align}\label{s3:hyperbolic-affine-support}
\L^{\mathbb H}(K)=\L(\hat{K}),
\end{align}
and hence it is $\operatorname{SA}(\mathbb H^n)$-invariant. The {\em hyperbolic $L_p$-affine surface area} of $K\subset \mathbb H^{n}$ is defined by 
\begin{align}\label{s3:p-affine-surface-area}
\operatorname{as}^{\mathbb H}_p(K):=\int_{\partial K} u^\frac{(1-p)n}{n+p}H_{n-1}(\~\k)^{\frac{p}{n+p}}dA, \quad \forall ~ p\neq -n,
\end{align}
and 
\begin{align}\label{s3:infty-affine-surface-area}
\operatorname{as}^{\mathbb H}_{\pm\infty}(K):=\int_{\partial K} u^{-n}H_{n-1}(\~\k)dA,
\end{align}
where $\~\k_i=V\k_i-V_{,\nu}$. Using $d\~A=H_{n-1}(\hat{\k})^\frac{1}{n+1}d\hat{A}=H_{n-1}(\~\k)^\frac{1}{n+1}dA$, we get
\begin{align}\label{s3:identity-affine-surface-area}
\operatorname{as}^{\mathbb H}_{p}(K)=&
\int_{\partial K} (\L^{\mathbb H})^\frac{(1-p)n}{n+p} d\~A=\operatorname{as}_p(\hat{K}), \quad \forall p\neq -n,
\end{align}
and
\begin{align*}
\operatorname{as}^{\mathbb H}_{\pm\infty}(K)=&
\int_{\partial K}(\L^{\mathbb H})^{-n} d\~A=\operatorname{as}_{\pm\infty}(\hat{K}).
\end{align*}
Hence, $\operatorname{as}_p(K)$ for all $p\neq -n$ and $\operatorname{as}^{\mathbb H}_{\pm\infty}(K)$ are $\operatorname{SL}(\mathbb H^n)$-invariant.

For any subset $\hat{K}\in \mathbb R^n$, the polar body of $\hat{K}$ in $\mathbb R^n$ with respect to the origin $o$ is defined by 
\begin{align*}
\hat{K}^{\circ}:=\{y\in \mathbb R^{n}~|~y\cdot x\leq 1, \forall x\in \hat{K}\}.
\end{align*}
The polar body $\hat{K}^{\circ}$ is an intersection of closed half-spaces that contain the origin, and thus it is a closed convex set that contains the origin. It follows from the definition that 
\begin{enumerate}[(i)]
	\item For any subsets $\hat{K}_0,\hat{K}_1$ in $\mathbb R^n$ such that $\hat{K}_0\subset \hat{K}_1$, then $\hat{K}_0^\circ \supset \hat{K}_1^\circ$.
	\item For any subset $\hat{K}$ in $\mathbb R^n$, $\hat{K} \subset \hat{K}^{\circ\circ}$.
\end{enumerate}
Moreover, if $\hat{K}$ is a convex body in $\mathbb R^n$ which contains the origin in its interior, then $\hat{K}^\circ$ is also a convex body that contains the origin in its interior, and $\hat{K}=\hat{K}^{\circ\circ}$, see \cite[Thm. 1.6.1]{Schneider}.

Another important affine isoperimetric inequality is the Blaschke-Santal\'o inequality.
\begin{thmE}[\cites{Blaschke1916, Santalo1949}]
	Let $\hat{K}\in \mathcal{K}^n$ with its centroid at the origin $o$. The Blaschke-Santal\'o inequality states that 
	\begin{align}\label{s3:BS-ineq}
	\Vol(\hat{K})\Vol(\hat{K}^{\circ}) \leq |\mathbb B^{n}|^2,
	\end{align}
	which is $\operatorname{SL}(n)$-invariant. Equality holds if and only if $\hat{K}$ is a centered ellipsoid. 
\end{thmE} 

In order to prove the Blaschke-Santal\'o type inequality in hyperbolic space, we first show that the following property of the hyperbolic polar body. 
\begin{prop}\label{s3:prop-space-form-polar-body}
	A smooth bounded domain $K$ in $\mathbb H^n$ and $p_0$ is an interior point of $K$. Let $K^{\circ}$ be the hyperbolic polar body of $K$ with respect to $p_0$ which is defined by \eqref{s1:polar-hyperbolic}. Then $$K^{\circ}=\pi_{p_0}^{-1}((\pi_{p_0}(K))^{\circ}).$$ 
	Moreover, if $K$ is static convex with respect to its hyperbolic centroid $p_0$, then $K^\circ$ is also static convex with respect to its hyperbolic centroid $p_0$.
\end{prop}
\begin{proof}
	By a hyperbolic translation, one may assume that $p_0=(1,0)\in \mathbb H^n$. For any $X=(\sqrt{1+|x|^2},x)$, $Y=(\sqrt{1+|y|^2},y)$ in $\mathbb H^n$, one has 
	\begin{align*}
	\cosh d(X,Y)=-\langle X,Y\rangle=-x\cdot y+\sqrt{1+|x|^2}\sqrt{1+|y|^2}.
	\end{align*}
	A direct calculation yields
	\begin{align*}
	\sqrt{1+|x|^2}\sqrt{1+|y|^2}=&\cosh d(X,p_0)\cosh d(Y,p_0).
	\end{align*}
	    Hence, $x\cdot y \leq 1$ is equivalent to
        $$
        \cosh d(X,p_0)\cosh d(Y,p_0) \leq \cosh d(X,Y)+1.
        $$ 
        Observe that $\pi:\mathbb H^n \ra \mathbb R^n$ is a diffeomorphism, so we get 
	\begin{align*}
	(\pi(K))^{\circ}=&\{x\in \mathbb R^n~|~x\cdot y\leq 1, \forall y\in \hat{K}\}\\
	=&\bigcap_{y\in \hat{K}}\{x\in \mathbb R^n~|~x\cdot y\leq 1\} \\
	=&\bigcap_{Y\in K}\pi(\{X\in \mathbb H^n ~|~\cosh d(X,p_0)\cos d(Y,p_0)\leq \cosh d(X,Y)+1\})\\
	=&\pi\bigg( \bigcap_{Y\in K} \{X\in \mathbb H^n ~|~\cosh d(X,p_0)\cos d(Y,p_0)\leq \cosh d(X,Y)+1\} \bigg)\\
	=&\pi(K^{\circ}).
	\end{align*}
    If $p_0=(1,0)$ is the hyperbolic centroid of $K$, then $\pi(K)$ is a compact convex body in $\mathbb R^n$ with its centroid at $o$ by Proposition \ref{s2:thm-centroid-characterization}. Thus, $(\pi(K))^\circ$ is also a compact convex body in $\mathbb R^n$ with its centroid at $o$. Using Proposition \ref{s2:thm-centroid-characterization} again, we conclude that $K^{\circ}=\pi^{-1}((\pi(K))^{\circ})$ is a smooth bounded domain which is static convex with respect to its hyperbolic centroid $p_0$.
\end{proof}

The radial function and support function of $\hat{K}\in \mathcal{K}_0^n$ with respect to the origin $o$ are defined by
\begin{align*}
\hat{r}(\hat{K},\t):=&\max\{t\in \mathbb R ~|~t \t\in  \hat{K}\}, \quad \t\in \mathbb S^{n-1},\\
\hat{h}(\hat{K},z):=&\max\{x\cdot z ~|~x\in  \hat{K}\},\quad z\in \mathbb S^{n-1}.
\end{align*}
In view of the polar duality, we have
\begin{align}\label{s3:polar-duality}
r(\hat{K}^{\circ},\t)=\frac{1}{\hat{h}(\hat{K},\t)}, \quad h(\hat{K}^{\circ},z)=\frac{1}{\hat{r}(\hat{K},z)}, \quad \text{for all $\t,z\in \mathbb S^{n-1}$}.
\end{align}
In particular,
$$
h(\hat{K},z)=\hat{u}(x), \quad \text{for $\mathcal{H}^{n-1}$-almost every $x\in \hat{\partial K}$},
$$ 
where $z$ is the unit outward normal at $x\in\hat{\partial K}$. In view of the polar duality \eqref{s3:polar-duality}, we have 
\begin{align*}
\hat{r}(\hat{K}^{\circ},\t)=\frac{1}{\hat{h}(\hat{K},\t)}=\frac{1}{\hat{u}(\hat\nu_{\hat{K}}^{-1}(\t))},\quad 
\hat{u}(\hat{\nu}_{\hat{K}^{\circ}}^{-1}(z))=\hat{h}(\hat{K}^{\circ},z)=\frac{1}{\hat{r}(\hat{K},z)}.
\end{align*}
It is clear that for any two subsets $K_1$, $K_2$ in $\mathbb H^n$ such that $K_1\subset K_2$, we have $K_1^\circ \supset K_2^\circ$; Moreover, for any subset $K$ in $\mathbb H^n$, $K \subset K^{\circ\circ}$.

\begin{prop} 
	If a smooth bounded domain $K$ in $\mathbb{H}^n$ is static convex with respect to an interior point $p_0$, then $K^\circ$ is also static convex with respect to $p_0$, and $K=K^{\circ\circ}$. 
		\begin{enumerate}[(i)]
			\item The radial function of the hyperbolic polar body satisfies
			\begin{align}\label{s3:identity-cosh-r}
			\cosh r^\circ=\frac{\cosh r}{u},
			\end{align}
			\item The hyperbolic special linear transformations and the hyperbolic polar bodies satisfy
			\begin{align}\label{s3:transform}
			(\phi_{L,0}(K))^{\circ}=\phi_{L^{-t},0}(K^{\circ}),
			\end{align}
			where $L^{-t}=(L^{-1})^t$ denotes the transpose of the inverse of $L\in \operatorname{SL}(n)$. 
			\item The hyperbolic affine support functions of $K$ and $K^\circ$ satisfy
			\begin{align}\label{s3:hyperbolic-support-function}
			\Lambda^{\mathbb H}(K,z)\Lambda^{\mathbb H}(K^{\circ},\t)=1,
			\end{align}
			where $z=\nu_{\hat{K}}^{-1}(x)$ and $\t=\frac{x}{|x|}$ for $x\in \hat{\partial K}$.
		\end{enumerate}		
\end{prop}
\begin{proof}
	Taking $\hat{K}=\pi_{p_0}(K)$, then the static convexity of $K$ with respect to $p_0$ is equivalent to the convexity of $\hat{K}$ by Corollary \ref{s2:cor-static-convex-equiv}. Then by Proposition \ref{s3:prop-space-form-polar-body}, we conclude that $K^\circ=\pi_{p_0}^{-1}((\pi_{p_0}(K))^\circ)$ is static convex with respect to $p_0$. Moreover, $K=K^{\circ\circ}$ follows immediately from $\hat{K}=\hat{K}^{\circ\circ}$.
    
    In view of \eqref{s2:radial-function-relation} and \eqref{s2:support-function}, \eqref{s3:identity-cosh-r} follows directly from 
\begin{align*}
\frac{\cosh r}{u}=\frac{(1+\hat{r}^2)^\frac{1}{2}}{\hat{u}}\(\frac{1+\hat{u}^2}{1+\hat{r}^2}\)^\frac{1}{2}=\(1+(\hat{r}^\circ)^2\)^\frac{1}{2}=\cosh r^\circ.
\end{align*}
Then \eqref{s3:transform} follows from $L(\hat{K}^{\circ})=L^{-t}(\hat{K}^{\circ})$ for all $L\in \operatorname{SL}(n)$. Finally, \eqref{s3:hyperbolic-support-function} follows from 
$$
\Lambda(\hat{K},z)\Lambda(\hat{K}^{\circ},\t)=1,
$$
where $\hat{K}$ and $\hat{K}^{\circ}$ are convex bodies in $\mathbb R^n$ with $o$ as an interior point which are the polar dual of each other, where $z=\nu_{\hat{K}}^{-1}(x)$ and $\t=\frac{x}{|x|}$ for $x\in \hat{\partial K}$, see \cite{Hug}.
\end{proof}

\begin{prop}
If a smooth bounded domain $K$ in $\mathbb{H}^n$ is strictly static convex with respect to its interior point $p_0$, and $p\neq -n$, then
\begin{align}\label{s3:identity-hyper-affine-surface-area}
    \operatorname{as}_p^{\mathbb{H}}(K)=\operatorname{as}_{\frac{n^2}{p}}^{\mathbb{H}}(K^{\circ}).
\end{align}
\end{prop}
\begin{proof}
	It follows from \cite[Thm. 3.2]{Hug} for $p>0$ and \cite[Cor. 3.1]{Werner-Ye2008} for general $p\neq -n$ that
	\begin{align*}
	\operatorname{as}_p(\hat{K})=\operatorname{as}_{\frac{n^2}{p}}(\hat{K}^{\circ}).
	\end{align*}
	Combining this with \eqref{s3:identity-affine-surface-area}, we obtain \eqref{s3:identity-hyper-affine-surface-area}.
	
\end{proof}

\section{Proofs of Theorems \ref{s1:main-thm-1}--\ref{s1:main-thm-3}}\label{sec:4}
\begin{proof}[Proof of Theorem \ref{s1:main-thm-1}]
	It follows from \eqref{s3:identity-affine-surface-area} that 
 	\begin{align*}
	\int_{\partial K}H_{n-1}(\~\k)^\frac{1}{n+1}dA=\operatorname{as}^{\mathbb H}_1(K)=\operatorname{as}_1(\hat{K}),
	\end{align*}
	where $\hat{K}=\pi_{p_0}(K)$. Using Theorem C and $\operatorname{Vol}(\hat{K})=\int_{K}V d\vol$, we obtain the desired inequality \eqref{s1:hyperbolic-affine-isop-ineq}. Equality holds if and only if $\hat{K}$ is an ellipsoid, and hence $K$ is a hyperbolic ellipsoid in $\mathbb H^n$. 
\end{proof}

\begin{proof}[Proof of Theorem \ref{s1:main-thm-2}]
	Assume that $K$ in $\mathbb H^n$ is static convex with respect to its hyperbolic centroid $p_0$. Then $\hat{K}=\pi_{p_0}(K)$ is a convex body with its centroid $o$ in $\mathbb R^{n}$ by Corollary \ref{s2:cor-static-convex-equiv} and Proposition \ref{s2:thm-centroid-characterization}. On the other hand, we have 
	\begin{align*}
	\int_{\partial K}u^\frac{(1-p)n}{n+p}H_{n-1}(\~\k)^\frac{p}{n+p}dA=\operatorname{as}^{\mathbb H}_p(K)=\operatorname{as}_p(\hat{K}).
	\end{align*}
    Using Theorem D and $\operatorname{Vol}(\hat{K})=\int_{K}V d\vol$, we obtain the desired inequalities \eqref{s1:Lp-affine-ineq-1}-\eqref{s1:Lp-affine-ineq-3}.
	Equality holds in \eqref{s1:Lp-affine-ineq-1} or \eqref{s1:Lp-affine-ineq-2} if and only if $\hat{K}$ is an centered ellipsoid in $\mathbb R^{n}$, which is equivalent to $K$ is a hyperbolic ellipsoid with its hyperbolic centroid $p_0$. 
\end{proof}

\begin{proof}[Proof of Theorem \ref{s1:main-thm-3}]
    It follows from Proposition \ref{s2:thm-centroid-characterization} that if $K$ has its hyperbolic centroid at $p_0$, then $\hat{K}:=\pi_{p_0}(K)$ is a convex body with centroid at the origin. Let $\hat{K}^{\circ} \subset \mathbb R^{n}$ be the polar body of $\hat{K}$ with respect to the origin. By Proposition \ref{s3:prop-space-form-polar-body}, we have $\pi_{p_0}(K^{\circ})=\hat{K}^{\circ}$. Thus, by \eqref{s2:weighted-volume} we get
    $$
    \Vol(\hat{K})=\int_{K}V d\vol, \quad \Vol(\hat{K}^{\circ})=\int_{K^{\circ}}V d\vol. 
    $$
    Using the Blaschke-Santal\'o inequality \eqref{s3:BS-ineq}, we obtain the desired inequality 
    $$
    \int_{K}V d\vol \cdot \int_{K^{\circ}} V d\vol \leq |\mathbb B^n|^2.
    $$ 
    The equality case in the Blaschke-Santal\'o inequality implies that $\hat{K}$ is an ellipsoid with its centroid $o$ in $\mathbb R^n$. Therefore, $K$ is a hyperbolic ellipsoid in $\mathbb H^{n}$ with its hyperbolic centroid at $p_0$.
\end{proof}

\end{document}